\documentclass{conm-p-l}
\usepackage{amssymb, stmaryrd}

\newtheorem{proposition}[equation]{Proposition}

\newcommand{\cM}{\mathcal{M}}

\makeatletter
\def\revddots{\mathinner{\mkern1mu\raise\p@
\vbox{\kern7\p@\hbox{.}}\mkern2mu
\raise4\p@\hbox{.}\mkern2mu\raise7\p@\hbox{.}\mkern1mu}}
\makeatother

\def\Ind{{\rm Ind\,}}

\def\E{\mathbb{E}}

\def\Q{\mathbb{Q}}
\def\R{\mathbb{R}}

\def\cM{\mathcal{M}}

\begin{document}

\title[Separating Vector Bundle Sections by Invariant Means]
{Separating Vector Bundle Sections by Invariant Means}

\author{Gestur \' Olafsson}
\address{Department of Mathematics, Louisiana State University, Baton Rouge,
LA 70803}
\curraddr{}
\email{olafsson@math.lsu.edu}
\thanks{G\' O: Research partially supported by NSF Grant DMS-1101337}

\author{Joseph A. Wolf}

\address{Department of Mathematics, University of California,
Berkeley, CA 94720--3840}
\curraddr{}
\email{jawolf@math.berkeley.edu}
\thanks{JAW: Research partially supported by the Simons Foundation}

\date{6 October, 2012}

\begin{abstract}We sharpen the construction of representation space in
the paper ``Principal Series Representations of Infinite Dimensional
Lie Groups II: Construction of Induced Representations''.  
We show that the principal series representation spaces constructed there,
are completions of spaces of sections of Hilbert bundles
rather than completions of quotient spaces of sections. 
\end{abstract}

\maketitle

This note is a continuation of \cite{W}, using the same notation.
We sharpen the construction of the representation spaces in 
\cite[\S\S 5B, 5C]{W} by proving that the bounded right uniformly
continuous sections of a homogeneous Hilbert space bundle $\E_\tau \to G/H$
(defined by a unitary representation $\tau$ of $H$) are separated by
means on $G/H$. 
\medskip

\centerline{\bf General Setting}
\smallskip

$G$ is a topological group, not necessarily locally compact, and $H$ is
a closed amenable subgroup.  $\tau$ is a unitary representation of $H$,
say on $E_\tau$, and $\E_\tau \to G/H$ is the associated 
homogeneous Hilbert space bundle.  The space
$RUC_b(G/H;\E_\tau)$ of bounded right uniformly continuous bounded sections
of $\E_\tau \to G/H$ consists of the right uniformly continuous bounded
functions $f: G \to E_\tau$ such that $f(xh) = \tau(h)^{-1}f(x)$ for
$x \in G$ and $h \in H$.  $G$ acts on it by $(\pi_\tau(x)f)(x') = f(x^{-1}x')$.
Since $\tau$ is unitary the pointwise norm $||f(xH)||$ is defined.  If
$\mu$ is a mean on $G/H$ we then have a seminorm on $RUC_b(G/H;\E_\tau)$
defined by $\nu_\mu(f) = \mu(||f||)$.  We denote the space of all means on 
$G/H$ by $\cM = \cM(G/H)$.
\smallskip

We use properties of means and amenability from \cite{Day1}, \cite{Day2}
and \cite{R}.

\begin{proposition}\label{5B}
If $0 \ne f \in RUC_b(G/H;\E_\tau)$ then there exists $\mu \in \cM =
\cM(G/H)$ such that $\nu_\mu(f) \ne 0$.  In other words, in 
\cite[Prop. 5.13 and Cor. 5.14]{W},
$\Gamma_{\cM}(G/H;\E_\tau)$ is the locally convex 
TVS completion of $RUC_b(G/H;\E_\tau)$.
\end{proposition}

\begin{proof} Let $f \in RUC_b(G/H;\E_\tau)$ be annihilated by 
all the seminorms $\nu_\mu$, $\mu \in \cM$.  Suppose that $f$ is not 
identically zero and choose $x \in G/H$ with $f(x) \ne 0$.  WE
can scale and assume $||f(x)|| = 1$.
Evaluation $\delta_x(\varphi) = \varphi(x)$ is a mean on $G$ and
$\delta_x(||f||) = 1$.
Now the compact convex set 
$S = \{\sigma \in \cM(G) \mid \sigma(||f||) = 1\}$ 
(weak$^*$ topology) 
is nonempty.  
Since $H$ is amenable it has a fixed point $\mu_f$ on $S$.  Now
$\mu_f$ is a mean on $G/H$ and the seminorm $\nu_{\mu_f}(f) = 1$.
\end{proof}
\medskip

\centerline{\bf Principal Series}
\smallskip

We specialize Proposition \ref{5B} to our setting where $G$ is a
real Lie group, e.g. $Sp(\infty;\R)$, and $P$ is a minimal
self--normalizing parabolic subgroup.  Then the amenably induced
representations $\Ind_P^G(\tau)$ of $G$ on the $\Gamma_{\cM}(G/P;\E_\tau)$,
in other words the general principal series representations of $G$,
do not require passage to quotient spaces of the $RUC_b(G/P;\E_\tau)$.
Further, the argument of \cite[Proposition 5.16]{W}, that
$\Ind_P^G(\tau)|_K = \Ind_M^K$ when the parabolic $P$ is flag-closed, 
is simplified because we need not compare quotient structures.
\medskip

\centerline{\bf Other Completions}
\smallskip

Here is a Fr\' echet space completion of $RUC_b(G/P;\E_\tau)$.  Note that
$G = \varinjlim G_n$ where the $G_n$ are real reductive groups defined over
the rational number field $\Q$ in a consistent way.  So we have the rational
group $G_\Q := \varinjlim G_{n,\Q}$.  The point is that the $G_{n,\Q}$ are
countable, so $G_\Q$ is countable, and the evaluations form a countable family
$\{\delta_{xP} \mid x \in G_\Q\}$ of means on $G/P$.  
If $f \in RUC_b(G/P;\E_\tau)$
and $||f||$ is annihilated by each of the ``rational'' seminorms
$\nu_{\delta_{xP}}$, the argument of Proposition \ref{5B} shows that $f = 0$.
The locally convex TVS structure of $RUC_b(G/P;\E_\tau)$, using only that
countable family of seminorms, defines a Fr\' echet space 
completion of $RUC_b(G/P;\E_\tau)$.  The action of $G_\Q$ extends by
continuity to this completion of $RUC_b(G/P;\E_\tau)$, but it is not
clear whether the the action of $G$ extends.
\smallskip

We enumerate $G_\Q$ by the positive integers to define a mean
$\mu = \sum_{m \geq 0} 2^{-m} \delta_{xP}^m$ on $G$. The corresponding seminorm
$\nu_\mu(f) = \sum_{m \geq 1} 2^{-m} ||f(x_m)||$ is a norm on
$RUC_b(G/P;\E_\tau)$. It defines a pre Hilbert space structure 
on $RUC_b(G/P;\E_\tau)$ by
$\langle f, h \rangle = \sum_{m \geq 1} 2^{-m} \langle f(x_m), h(x_m) \rangle$.
Again, the action of $G$ on $RUC_b(G/P;\E_\tau)$ does not appear to 
extend by continuity to the corresponding Hilbert space completion.

\end{document}